\documentclass[12pt]{amsart}
\usepackage{hyperref}
\usepackage{amsfonts,mathrsfs,bbm,rawfonts,amsmath,amssymb}
\usepackage{fullpage, setspace,geometry,booktabs}
\usepackage{graphics, color}
\usepackage{tikz-cd}
\usepackage{float}
\newtheorem{thm}{Theorem}[section]
\newtheorem{lemma}[thm]{Lemma}
\newtheorem*{lemma*}{Lemma}
\newtheorem{prop}[thm]{Proposition}

\newtheorem{cor}[thm]{Corollary}

\newtheorem*{conj*}{Conjecture}

\newtheorem{rmk}[thm]{Remark}
\newtheorem*{rmk*}{Remark}

\numberwithin{equation}{section}
 \newcommand{\al}{\alpha}
 \newcommand{\be}{\beta}
 \newcommand{\ld}{\lambda}
 
 \newcommand{\de}{\delta}
 
 \newcommand{\ep}{\epsilon}
 \newcommand{\Si}{\Sigma}

 \newcommand{\G}{\mathscr{G}}

 \newcommand{\R}{\mathbb{R}}

 \newcommand{\norm}[1]{\Vert#1\Vert}
 
 \renewcommand{\t}{\tilde}
 \renewcommand{\b}{\bar}
 
 \newcommand{\p}{\partial}
 \newcommand{\n}{\nabla}

 \def\<{\langle} \def\>{\rangle}
 \def\({\left(} \def\){\right)}

 \DeclareMathOperator{\tr}{tr}
 
 \DeclareMathOperator{\Ric}{Ric}
 \DeclareMathOperator{\vol}{vol}

\title
{Correspondence between Mean Curvature Flow and Harmonic-Ricci Flow}

\author[T. Ren and C. Song]
{Tianyin Ren and Chong Song}

\address{Tianyin Ren
\newline\indent
School of Mathematical Sciences, Xiamen University
\newline\indent
Fujian, 361005, P.R. China
}
\email{tianyin.ren@xmu.edu.cn}

\address{Chong Song \newline\indent School of Mathematical Sciences, Xiamen University
\newline\indent
Fujian, 361005, P.R. China
}
\email{songchong@xmu.edu.cn}

\date{May 16, 2026}

\subjclass[2020]{53E10, 53E20, 58E20}

\begin{document}
\setlength{\baselineskip}{15pt}

\begin{abstract}
	In this paper, we observe that the (spacelike) mean curvature flow of a submanifold in a (pseudo-)Euclidean space is equivalent to a harmonic-Ricci flow with coupling constant $\alpha=-1$ (or $+1$), for the corresponding Gauss map and the induced metric. The solitons of these two flows are also equivalent. As an application, we get a monotonicity formula for the spacelike mean curvature flow.
\end{abstract}
\maketitle

\section{Introduction}

In 1970, Ruh and Vilms~\cite{RV1970} observed that the tension field of the Gauss map of a submanifold in a Euclidean space can be identified with the covariant derivative of the mean curvature vector field. Therefore, the Gauss map of a submanifold with parallel mean curvature vector field, including minimal submanifolds, is harmonic. Following this line, Wang~\cite{Wang2003} showed a parabolic analog of Ruh-Vilms' theorem. Namely, the Gauss map of a \emph{mean curvature flow (MCF for short)} in a Euclidean space satisfies the harmonic map heat flow.

In this paper, we go one step further and show that the Gauss map together with the induced metric of a MCF forms a closed system, i.e. the so-called \emph{harmonic-Ricci flow (HRF for short)}. Conversely, given the Gauss map and induced metric of an initial submanifold, we can first solve the corresponding HRF and reconstruct the solution to the MCF. This equivalence to the HRF also holds true for \emph{spacelike mean curvature flow (SMCF for short)} in pseudo-Euclidean spaces, with a sign difference in the coupling constant.

To state our main results, let us first fix some notations. Let $\Si$ be a compact or complete non-compact $n$-dimensional manifold. Let $M$ be an $n+k$ dimensional flat ambient space and $\ep\in \{-1, +1\}$ be a sign indicator. We consider two primary cases:
\begin{itemize}
    \item $\Sigma$ is a submanifold in Euclidean space $M = \R^{n+k}$, and we set $\ep = 1$.
    \item $\Sigma$ is a spacelike submanifold in pseudo-Euclidean space $M = \R^{n,k}$, and we set $\ep = -1$.
\end{itemize}
For simplicity, we will always use MCF to refer to both cases unless otherwise stated.

\begin{thm}\label{thm:main}
Suppose $X_0:\Si \to M$ is a immersed (spacelike) submanifold with Gauss map $\rho_0$ and pull-back metric $g_0$. Then the MCF
\begin{equation}\label{eq:mcf}
	\begin{cases}
		\partial_t X = H, \\
		X(0) = X_{0}.
	\end{cases}
\end{equation}
and the HRF
\begin{equation}\label{eq:hrf-gauss}
	\begin{cases}
        \partial_t \rho = \tau_g(\rho), \\
		\partial_t g = -2\Ric_g + 2 \alpha ~d\rho\otimes d\rho, \\
        g(0) = g_{0}, \quad \rho(0) = \rho_{0},
	\end{cases}
\end{equation}
with coupling constant $\alpha=-\ep$, are equivalent in the following sense.

Suppose $X(t)$ is the unique solution to the MCF~\eqref{eq:mcf} and $(\rho(t), g(t))$ is the unique solution to the HRF~\eqref{eq:hrf-gauss}. Then
\begin{enumerate}
    \item The gauss map and induced metric of $X(t)$ is exactly $(\rho(t), g(t))$.
    \item The MCF $X(t)$ can be directly constructed form the HRF $(\rho(t), g(t))$.
\end{enumerate}
\end{thm}

\begin{rmk}
\begin{enumerate}
    \item The correspondence between MCF and HRF in Theorem~\ref{thm:main} remains valid even without assuming uniqueness of the solutions. We impose the uniqueness assumption here to simplify the exposition and to avoid addressing uniqueness issues, which are not the main focus of the present paper. In particular, uniqueness is standard when $\Si$ is closed.
    \item The precise meaning of item~(2) in Theorem~\ref{thm:main} is explained in Section~2.2, where we give an explicit method of reconstructing the MCF.
\end{enumerate}
\end{rmk}

Theorem~\ref{thm:main} shows that the HRF naturally links perhaps the most important three geometric flows, i.e. the harmonic map heat flow, the Ricci flow and the mean curvature flow. Now any results of the MCF would imply a parallel result of the corresponding HRF into Grassmannian manifolds. Conversely, the HRF can be viewed as an intrinsic version of the MCF, and provides a different perspective to the analysis of singularities, especially for higher codimensional cases.

For example, the following corollary is now a direct consequence of Theorem~\ref{thm:main} and well-known facts of MCF.

\begin{cor}\label{c:main}
Under the assumption of Theorem~\ref{thm:main}, we have
\begin{enumerate}
  \item The Gauss-Codazzi equations are preserved under the HRF~\eqref{eq:hrf-gauss}.
  \item When $\Si$ is spacelike, the HRF~\eqref{eq:hrf-gauss} exists for all time.
  \item When $\Si$ is closed and $M$ is Euclidean, the HRF~\eqref{eq:hrf-gauss} blows up at a finite time $T$, which is characterized by
    \begin{equation}\label{e:maximal-time}
        \lim_{t\to T^-}\norm{d\rho(t)}_{L^\infty(\Si)}=+\infty.
    \end{equation}
\end{enumerate}
\end{cor}

Another consequence is the correspondence between MCF solitons and HRF gradient solitons.

\begin{cor}\label{c:soliton}
A submanifold $X:\Si\to M$ is a soliton of the MCF~\eqref{eq:mcf} if and only if the corresponding pair  $(\rho, g)$ consisting of its Gauss map and induced metric is a gradient soliton of the HRF~\eqref{eq:hrf-gauss}.
\end{cor}

The equivalence between MCF and HRF in Theorem~\ref{thm:main} also holds for a volume-preserving version. When the ambient space is a general Riemannian manifold or pseudo-Riemannian manifold, the Gauss map and the induced metric of the MCF or SMCF satisfies a harmonic-Ricci-type system with additional terms involving the curvature of the ambient manifold, see the Appendix.

The HRF was introduced by M\"uller~\cite{Muller2009, Muller2012} as a natural generalization of the Ricci flow, with some special cases previously studied by List~\cite{List2008} and Lott~\cite{Lott2007}. The local existence and uniqueness of HRF was established in~\cite{Muller2009}. Many techniques and important features of the Ricci flow carry over almost directly to the HRF, but only when the coupling constant $\alpha$ is positive (or more generally, when $\alpha(t)$ is a positive function). In fact, in some special situations, the HRF behaves less singular than the Ricci flow or the harmonic map heat flow. For example, the volume-preserving HRF on surfaces exists for all time if the coupling constant is sufficiently large~\cite{BR2017}.

On the other hand, HRF with negative coupling constant is much harder to study due to lack of monotonicity properties. In view of Theorem~\ref{thm:main}, this explains the key difference between the MCF in Euclidean spaces and SMCF in pseudo-Euclidean spaces.

Finally, as an application, we obtain a monotonicity formula for SMCF, which is a direct consequence of the monotonicity formula for HRF with pocitive coupling constant in~\cite{Muller2012}.

\begin{thm}\label{thm:monotonicity}
Let $X(t):\Si^n\to M$ be a MCF and $f$ solve the (intrinsic) adjoint heat equation under this flow:
\begin{equation} \label{eq:adjoint_heat}
    \partial_t f = -\Delta f + |\nabla f|^2 + \alpha |H|^2 .
\end{equation}
Then the $\mathcal{F}$-functional
\[ \mathcal{F}(t) := \int_\Si\(|\nabla f|^2-\alpha |H|^2\)e^{-f}dV_g \]
satisfies
\[ \frac{d}{dt}\mathcal{F} = \int_\Si\(2|\nabla^2 f - \alpha\<H, A\>|^2 + 2\alpha|\nabla H -\<A, \nabla f\>|^2\)e^{-f}dV_g, \]
provided the above integrals are finite.

In particular, along the SMCF where $\alpha=1$, $\mathcal {F}(t)$ is non-decreasing.
\end{thm}

\

The rest of the paper is organized as follows. First we rewrite the Gauss-Codazzi equations by differentials of the Gauss map in Section 2. Then we prove the main Theorem~\ref{thm:main} and Corollary~\ref{c:main} in Section 3. The proof of Corollary~\ref{c:soliton} and Theorem~\ref{thm:monotonicity} is given in Section 4. In the Appendix, we show a volume-preserving version of Theorem~\ref{thm:main} and derive the corresponding harmonic-Ricci-type equations in a curved ambient space.

\section{Gauss maps and structure equations of submanifolds}

Let $\Si$ be an $n$-dimensional manifold, $M$ be an $n+k$ dimensional ambient space, which is either Euclidean space $\R^{n+k}$ or pseudo-Euclidean space $\R^{n,k}$.

Suppose $X:\Si\to M$ is a (spacelike) immersed submanifold. Denote its image by $\b{\Sigma} = X(\Sigma)$ and its normal bundle by
 $$ N\Sigma = X^*(T\b{\Sigma}^\bot).$$
Note that the induced metric on $N\Si$ is positive definite if $M=\R^{n+k}$ and negative definite if $M=\R^{n,k}$.

Let $\{e_i\}_{i=1}^n$ be a set of local orthonormal frame of $T\Si$ and $\{\de^i\}$ be its dual. Let $\{\bar{e}_i=dX(e_i)\}$ be the corresponding orthonormal frame of $X^*(T\bar{\Si})$ and $\{\bar{\de}^i\}$ be its dual. Let $\{\nu_{\al}\}_{\al=1}^k$ be local orthonormal frame of $N\Si$. The pull-back map
\begin{equation}\label{e:pull-back}
    X^*:T_{X(x)}^*\bar{\Si}\to T_x^*\Si, \quad \{\bar{\de}^i\}_{i=1}^n \mapsto \{\de^i\}_{i=1}^n
\end{equation}
maps a section of $X^*(T^*\bar{\Si})$ isometrically to a section of $T^*\Si$.

Set the sign indicator $\ep=+1$ if $M=\R^{n+k}$ is Euclidean space and $\ep=-1$ if $M=\R^{n,k}$ and $X$ is spacelike. Recall that since $M$ is flat, the Gauss equation reads
\begin{equation}\label{e:gauss}
  R_{ijkl} = \ep(h^\al_{ik}h^\al_{jl} - h_{il}^\al h_{jk}^\al).
\end{equation}
Taking trace,
\begin{equation}\label{e:gauss-ricci}
  \Ric_{ij}=\ep(H^\al h_{ij}^\al - h_{ik}^\al h_{jk}^\al).
\end{equation}
Taking another trace yields
\begin{equation}\label{e:gauss-scalar}
    R=\ep(|H^\al|^2-h_{ik}^\al h_{ik}^\al).
\end{equation}
The Codazzi equation is
\begin{equation}\label{e:codazzi}
  \n_i h_{jk} = \n_j h_{ik}.
\end{equation}
The Ricci equation is
\begin{equation}\label{e:ricci}
    R_{ij\al\be}= h_{ik}^\al h_{jk}^\be - h_{jk}^\al h_{ik}^\be.
\end{equation}

Now, consider the Gauss map $\rho:\Si \to G$ of $X$, where the Grassmannian manifold is defined by
\[ G:= \{ \text{oriented k-dimensional (spacelike) linear subspaces in } M \}. \]
The Pl\"ucker embedding $i: G \hookrightarrow \wedge^{n}M$ gives us a canonical metric on $G$ which induced a canonical connection  $\nabla^G$. When $k=1$ and $\Si$ is a hypersurface, $G$ is the standard sphere $\mathbb{S}^n$ if $M$ is Euclidean, while $G$ is the standard hyperbolic space $\mathbb{H}^n$ if $M$ is the Minkowski space $\mathbb{R}^{n,1}$.

Denote the canonical tautological bundle and cotautological bundle over $G$ by $\G^{\top}$ and $\G^{\perp}$, with canonical connections $\b{\nabla}^{\top}$  and $\b{\nabla}^{\bot}$, respectively. There is a bundle isomorphism (see for example \cite{Song}),
\begin{equation}\label{eq:bun-iso-totau}
   (TG, \nabla^{G}) \cong (\G^{\perp} \otimes \G^{\top *}, \b{\nabla}^{\bot}\oplus\b{\nabla}^{\top}).
\end{equation}  
Since $\rho^*\G^\top = X^*T\b{\Si}$ and $\rho^*\G^\bot = X^*N\b{\Si}$, we have an canonical bundle isometry
$$ \rho^*(TG, \nabla^G) \cong X^*(N\b{\Sigma}\otimes T^*\b{\Sigma}, \b{\nabla}^{\bot}\oplus\b{\nabla}^{\top}), $$
such that  the differential
$$ d\rho \in \Gamma(\rho^*TG\otimes T^*\Sigma) \cong \Gamma(N\Sigma\otimes X^*(T^*\b{\Sigma})\otimes T^*\Sigma).$$
can be identified with the second fundamental form $A$ by
\begin{equation}\label{eq:sec-gauss-0}
A=X^*d\rho.
\end{equation}
More explicitly, if we write $A= h_{ij}^\al \nu_{\al}\otimes \de^i\otimes \de^j $ and $d\rho = \nabla_i\rho_{j}^\al \nu_{\al}\otimes \de^i\otimes \bar{\de}^{j}$ in local frames~\eqref{e:pull-back}, then
\[ h_{ij}^\al =  \nabla_i\rho_{j}^\al. \]

Consequently, the Gauss equations (\ref{e:gauss}-\ref{e:gauss-scalar}) becomes
\begin{eqnarray}
  R_{ijkl} &=& \ep(\nabla_i\rho_{k}^\al\nabla_j\rho_{l}^\al - \nabla_i\rho_{l}^\al\nabla_j\rho_{k}^\al), \label{e:gauss-1} \\
  \Ric_{ij} &=& \ep(\nabla_k\rho_{k}^\al\nabla_i\rho_{j}^\al - \nabla_i\rho_{k}^\al\nabla_j\rho_{k}^\al), \label{e:gauss-ricci-1} \\
  R &=& \ep(|\nabla_k\rho_{k}^\al|^2 - \nabla_i\rho_{k}^\al\nabla_i\rho_{k}^\al). \label{e:gauss-scalar-1}
\end{eqnarray}
The Codazzi equation (\ref{e:codazzi}) just becomes the symmetry of the Hessian of $\rho$, i.e.
\begin{equation}\label{e:codazzi-rho}
  \n_i\n_j \rho = \n_j \n_i \rho.
\end{equation}
It follows that
\begin{equation}\label{e:tau-rho}
  \n H = \tau_g(\rho), 
\end{equation}
where $\tau_g(\rho)$ is the tension field of $\rho$. The Ricci equation (\ref{e:ricci}) becomes
\begin{equation}\label{e:ricci-1}
   R_{ij\al\be}=\n_i\rho_{k}^\al \n_j\rho_{k}^\be - \n_j\rho_{k}^\al \n_i\rho_{k}^\be.
\end{equation}

\section{Equivalence between MCF and HRF}

\subsection{From MCF to HRF}

Let $X:[0,T)\times \Si \to M$ be a MCF satisfying
\begin{equation}\label{e:mcf}
  \begin{cases}
  \p_t X = H,\\
  X(0)=X_0.
\end{cases}
\end{equation}

For $t\in [0, T)$, let $\rho(t):\Si \to G$ be the Gauss map of $X(t)$. It is well-known~\cite{Wang2003} that under the MCF (\ref{e:mcf}), $\rho$ satisfies the harmonic map heat flow
\begin{equation}\label{e:rho}
  \p_t \rho = \tau_{g}(\rho),
\end{equation}
where $g(t)$ is the time-dependent pull-back metric. Moreover, $g(t)$ satisfies the evolution equation (see for example~\cite{Hui11984})
\begin{equation}\label{e:g}
  \p_t g = -2\<H, A\>.
\end{equation}
Note that if $M$ is pseudo-Euclidean and $X_0$ is spacelike, then the corresponding MCF $X$ is always spacelike~\cite{Ecker1997}.
Denoting the tensor
\[ d\rho\otimes d\rho=\n_i\rho_{k}^\al \n_j\rho_{k}^\al, \]
and inserting the Gauss equation (\ref{e:gauss-ricci-1}) into \eqref{e:g}, we have
\begin{equation}\label{e:g1}
  \p_t g = -2 \Ric - 2\ep~d\rho\otimes d\rho.
\end{equation}

Combining \eqref{e:rho} and \eqref{e:g1}, we get the following proposition which proves item (1) of Theorem~\ref{thm:main}.

\begin{prop}\label{p:MCF2HRF}
The Gauss map $\rho(t)$ and the pull-back metric $g(t)$ of MCF~\eqref{e:mcf} satisfy
\begin{equation}\label{e:mcf-gauss}
  \begin{cases}
    \p_t \rho = \tau_{g}(\rho),\\
    \p_t g = -2 Ric - 2\ep~d\rho\otimes d\rho,
  \end{cases}
\end{equation}
which is exactly the HRF with coupling constant $\al=-\ep$.
\end{prop}

\subsection{Reconstruction of MCF from HRF}

Next we show that given an initial immersion $X_0:\Sigma\to M$ with Gauss map $\rho_0$ and induced metric $g_0$, we may first solve the HRF \eqref{e:mcf-gauss} with initial data $\rho_0$ and $g_0$, and then reconstruct a solution $X(t)$ to the MCF~\eqref{e:mcf} with initial data $X_0$.

The standard verification of this correspondence typically relies on demonstrating that the Gauss-Codazzi integrability conditions are preserved along the HRF via the maximum principle, followed by a reconstruction of the immersion. While this analytic approach involves intricate curvature tensor calculations (see, e.g., \cite{Zhang2009}), we provide a more transparent and geometric construction that bypasses these computations by directly reconstructing the immersion through parallel transport in the space-time bundles.

Let $(\rho, g)$ be a solution to the HRF~\eqref{e:mcf-gauss} on the time interval $I=[0,T)$ with initial data $(\rho_0,g_0)$. We may define a Riemannian metric $\t{g}=dt^2+g(t)$ on the product space $\t{\Si}=\Si\times I$. Let $\t{\n}$ be the Levi-Civita connection induced by $\t{g}$. The tangent bundle $T\t{\Si}$ splits into a direct product $\mathcal{H}\oplus \mathcal{I}$, where
$$\mathcal{H}=\{u\in T\t{\Si}|dt(u)=0\}$$
is the `spatial' tangent bundle and $\mathcal{I}$ is the trivial line bundle spanned by $\p_t$. Then $\t{\n}$ restricts to a `spatial' connection $\n=\pi_{\mathcal{H}}\circ\t{\n}$ on $\mathcal{H}$, which coincides with the Levi-Civita connection of $g(t)$ at each time slice $\Si\times\{t\}$. For more details on the geometry of time-dependent bundles, we refer to Chapter 2 of Baker~\cite{Baker2010}.

On the other hand, there is a pull-back bundle $\rho^*\G^\top$ on $\t{\Si}$ with pull-back metric $\rho^*\b{g}$ and pull-back connection $\rho^*\b{\n}^\top$, where $\b{g}$ and $\b{\n}^\top$ is the canonical metric and connection on $\G^\top$ induced by the standard one on $M$.

The key step of our proof is to construct a bundle isometry between $\mathcal{H}$ and $\rho^*\G^\top$, which plays the role of the soldering map $dX$.

\begin{lemma}\label{l:bundle-isometry}
There exists a bundle isometry
 $$\Phi: \mathcal{H} \to \rho^* \G^{\top} $$
such that
\begin{equation}\label{e:preserve-connection}
  \nabla_{\partial_t} = \Phi^{*}(\rho^{*}\b{\nabla}^{\top}_{\partial_t}).
\end{equation}
\end{lemma}

\begin{proof}
Let $\{\partial_i\}_{i=1}^n$ be a local frame of $\mathcal{H}$. Since the spatial connection $\n$ is the restriction of the Levi-Civita connection $\t{\n}$, the connection coefficients along the $t$-direction are given by
\begin{equation}\label{e:connection-int}
\nabla_{t} \partial_i
= \Gamma_{ti}^j\partial_j
= \frac{1}{2}\t{g}^{jl}\bigl(\partial_t \t{g}_{il} + \partial_i \t{g}_{tl} - \partial_l \t{g}_{ti}\bigr)\partial_j
= \frac{1}{2}g^{jl}\partial_t g_{il}\,\partial_j.
\end{equation}

On the other hand, given an adapted orthonormal local frame
$$\{e_1,\dots,e_n;\nu_1,\dots,\nu_k\}$$
on $\rho^*\G^{\top}\oplus \rho^*\G^{\bot}$, we may calculate the connection coefficients of $\rho^*\b{\n}^\top$ as follows. Writing
$$\rho(t)=e_1\wedge\cdots\wedge e_n,$$
we have
\begin{align*}
\partial_t\rho
&= \bar\nabla_t e_1\wedge\cdots\wedge e_n + \cdots + e_1\wedge\cdots\wedge \bar\nabla_t e_n \\
&= \langle \bar\nabla_t^{\perp} e_i,\nu_\alpha\rangle\, \nu_\alpha\otimes e_i^*.
\end{align*}
In the last identity, we used the bundle identification
$$TG=\G^{\perp}\otimes (\G^{\top})^*.$$
Therefore,
$$\partial_t\rho(e_i)=\bar\nabla_t^{\perp}e_i.$$
It follows that
\begin{equation}\label{e:connection-ext}
\bar\nabla^{\top}_{t}e_i
= \bar\Gamma_{ti}^j e_j
= \bar\nabla_t e_i - \partial_t\rho(e_i).
\end{equation}

Now we can construct the bundle isometry by matching parallel frames on $\mathcal{H}$ and $\rho^*\G^\top$.  Given any point $x_0\in\Sigma$, let $\{x^i\}$ be the local normal coordinates with respect to $g_0$, and choose an adapted local frame along the curve $\gamma(t)=(x_0,t), t\in I$ such that
\[ \varepsilon_i(0):=\partial_i, \qquad \bar e_i(0) = e_i(0) =dX_0(\partial_i).\]
Next we may parallel transport $\varepsilon_i(0)$ in $\mathcal{H}$ with respect to the intrinsic connection $\n$, and parallel transport $e_i(0)$ in $\rho^*\G^\top$ with respect to the extrinsic connection $\rho^*\b{\n}^\top$. Then the bundle isometry is given by sending $\varepsilon_i(t)$ to $e_i(t)$.

More precisely, let the connection coefficients $\Gamma_{ti}^j$ and $\b{\Gamma}_{ti}^j$ be given by \eqref{e:connection-int} and \eqref{e:connection-ext}, respectively. Write
$$\varepsilon_i(t)=a_i^j(t)\partial_j, \qquad \bar e_i(t)=b_i^j(t)e_j(t),$$
where the coefficients $a_i^j$ and $b_i^j$ solve the ODEs
$$\nabla_t\varepsilon_i(t)=\left(\partial_t a_i^j + a_i^k \Gamma_{tk}^j\right)\partial_j=0,$$
and
$$\b{\nabla}_t^{\top}\bar e_i(t)=(\partial_t b_i^j + b_i^k\bar\Gamma_{tk}^j)e_j=0,$$
with initial values
$$a_i^j(0)=b_i^j(0)=\delta_i^j.$$
Then the desired bundle isometry is given by
$$\Phi(t)=\varepsilon_i(t)^*\otimes \bar e_i(t):\mathcal{H}|_t\to \rho^*\G^\top|_t, \quad \forall t\in I,$$
which clearly satisfies (\ref{e:preserve-connection}).
\end{proof}

With the bundle isometry $\Phi$, we can now recover the mean curvature vector by
\[ \tr_{\Phi} d\rho=g^{ij}\Phi_j^k \n_i\rho_{k}^{\alpha}.\]
Then we solve the ODE
\begin{equation}\label{e:ODE}
  \begin{cases}
    \partial_t\tilde X=\tr_{\Phi}d\rho, \\
    \tilde X(0)=X_0.
  \end{cases}
\end{equation}

\begin{lemma}\label{l:HRF2MCF}
  The solution $\tilde X(t)$ of ODE~\eqref{e:ODE} is exactly the unique solution to the MCF~\eqref{e:mcf} starting from $X_0$.
\end{lemma}

\begin{proof}
Suppose $X(t)$ is the unique solution to the MCF~\eqref{e:mcf} starting from $X_0$. Then by Proposition~\ref{p:MCF2HRF}, its Gauss map and induced metric satisfy the HRF~\eqref{e:mcf-gauss}. By the assumption of uniqueness of HRF, it must coincide with $(\rho(t), g(t))$.


In view of the correspondence between $X(t)$ and $(\rho(t), g(t))$, we may identify $\rho^*\G^{\top}$ with $X^*T\bar{\Si}$. Moreover, we have the bundle isometry
$dX:\mathcal{H}\to X^*T\bar{\Si}.$
The spatial connection $\nabla$ on $\mathcal{H}$ is exactly the pull-back of the induced connection on the tangent bundle $T\bar{\Si}$, which is in turn identified with $\rho^*\b{\nabla}^{\top}$.
Therefore, the soldering map $dX$ gives a bundle isometry
$$dX:\mathcal{H}\longrightarrow \rho^*\G^{\top}$$ 
which preserves the connections, i.e.
$$\nabla  = dX^{*}(\rho^{*}\b{\nabla}^{\top}).$$
In particular, $dX$ preserves the parallel transport in the $t$-direction. Namely, it satisfies the same equation (\ref{e:preserve-connection}) as $\Phi$. Since $dX(0)=\Phi(0)=dX_0$, the uniqueness of ODE implies that 
\[ \Phi=dX \] 
for all $(x,t)\in \Si\times I$. It follows that $\tr_\Phi d\rho$ is just the mean curvature vector $H$ of $X$. Therefore, the map $\t{X}(t)$ and $X(t)$ satisfy the same ODE~\eqref{e:ODE} and must be identical.
\end{proof}

Now the main Theorem~\ref{thm:main} is a direct consequence of Proposition~\ref{p:MCF2HRF} and Lemma~\ref{l:HRF2MCF}. So we omit the proof.

\begin{proof}[Proof of Corollary~\ref{c:main}]
Item (1) is a direct consequence of Theorem~\ref{thm:main}. Item (2) follows from the long-time existence of SMCF~\cite{Ecker1997, LL2021} and Theorem~\ref{thm:main}.

For item (3), it is well-known that the MCF for a closed manifold in Euclidean space develop singularities at a finite time $T$ which is characterized by
\[ \lim_{t\to T^-}\norm{A}_{L^\infty(\Si)} = +\infty.\]
Since $|d\rho|=|A|$, it follows that the corresponding HRF also blows-up at $T$  such that~\eqref{e:maximal-time} holds.
\end{proof}

\section{Applications}

\subsection{Correspondence of solitons}

In this subsection, we prove the correspondence of the solitons of MCF and HRF.

Recall that a soliton to the MCF is a submanifold $X:\Si\to M$ such that
\begin{equation}\label{eq:MCF-soliton-shrinking}
  H+\ld X^{\bot}=0,
\end{equation}
where $\ld=1$ or $-1$ and $X$ is called a shrinking or expanding soliton respectively, or
\begin{equation}\label{eq:MCF-soliton-translating}
  H=V^{\bot},
\end{equation}
where $V$ is a constant vector field and $X$ is called a translating soliton.

A gradient soliton of the HRF is a pair $(\rho, g)$ such that
\begin{equation}\label{eq:HRF-soliton}
    \begin{cases}
        \Ric - \alpha ~d\rho \otimes d\rho + \nabla^2 f = \lambda g \\
        \Delta \rho = \langle \nabla f, \nabla \rho \rangle,
    \end{cases}
\end{equation}
where $f$ is a potential function, $\ld = -1, 0, 1$ and $(\rho, g)$ is called a gradient shrinking, steady, expanding soliton respectively.

Let us restate Corollary~\ref{c:main}, i.e. the correspondence of solitons in a more explicit form.
\begin{cor}
A submanifold $X:\Si\to M$ is a soliton of the MCF, i.e. satisfies \eqref{eq:MCF-soliton-shrinking} or \eqref{eq:MCF-soliton-translating}, if and only if the pair $(\rho, g)$ consisting of its Gauss map and induced metric is a gradient soliton of the corresponding HRF satisfying \eqref{eq:HRF-soliton}. The corresponding potential functions $f$ and constants $\ld$ are given in Table 1.
\end{cor}

\begin{table}[htbp]
    \centering
    \caption{Soliton Correspondence ($\alpha = -\epsilon$)}
    \label{tab:dictionary}
    \vspace{0.3cm}
    \renewcommand{\arraystretch}{1.5}
    \begin{tabular}{@{}llll@{}}
        \toprule
        \textbf{Soliton Type} & \textbf{MCF Equation} & \textbf{Potential $f$} & \textbf{HRF Constant $\lambda$} \\ \midrule
        {Shrinker} & ${H} +  X^\perp = 0$ & $f = \frac{ |X|^2}{2}$ & $\lambda = 1$ \\
        {Expander} & ${H} - X^\perp = 0$ & $f = -\frac{ |X|^2}{2}$ & $\lambda = -1$ \\
        {Translator} & ${H} = V^\perp$ & $f = - \langle X, V \rangle$ & $\lambda = 0$ \\\bottomrule
    \end{tabular}
\end{table}

\begin{proof}
By the Gauss equation~\eqref{e:gauss-ricci-1}, we have
\begin{equation} \label{eq:gauss_id}
    \Ric - \alpha \nabla \rho \otimes \nabla \rho = \langle H, A \rangle.
\end{equation}
Here $\<\cdot, \cdot\>$ denotes the standard inner-product of the ambient space $M$.

\emph{Part 1: MCF Soliton $\implies$ HRF Soliton}\\
First assume $X$ is an MCF shrinking or expanding soliton, which satisfies~\eqref{eq:MCF-soliton-shrinking} with $\ld=1$ or $-1$. Set $f= \frac{\ld|X|^2}{2}$. We compute
\[ \n f = \ld\<\n X, X\> =  \ld X^\top\]
and
\begin{equation}\label{eq:hessian-f-1}
  \n^2 f = \ld(\<\n^2 X, X\> + \<\n X, \n X\>)=\ld\<A, X^\bot\> + \ld g.
\end{equation}
It follows from \eqref{eq:gauss_id}, \eqref{eq:hessian-f-1} and \eqref{eq:MCF-soliton-shrinking} that
\[ \Ric - \alpha \nabla \rho \otimes \nabla \rho + \n^2 f = \<H+\ld X^\bot, A\> + \ld g = \ld g.\]
This verifies the first equation of the HRF soliton equation~\eqref{eq:HRF-soliton}. For the Gauss map $\rho$, we have
\[ \nabla \rho = A, \]
and
\[ \<\n f, \n \rho\> = \<\ld X^\top, A\> =  \ld A( X^\top, \cdot). \]
On the other hand,
\[ \Delta \rho  = \n^\bot H = -\ld\n^\bot X^\bot = \ld A(X^\top, \cdot).\]
This verifies the second equation of~\eqref{eq:HRF-soliton}.

Next, assume $X$ is an MCF translating soliton, which satisfies~\eqref{eq:MCF-soliton-translating} where $V$ is a constant vector field. Set $f=-\<X, V\>$. We compute
\[ \n f = -V^\top \]
and
\[ \n^2 f= -\<A, V^\bot\>.\]
Adding this to \eqref{eq:gauss_id} gives
\[ \Ric - \alpha \nabla \rho \otimes \nabla \rho + \n^2 f = \<H - V^\bot, A\> =0.\]
For the Gauss map, we have
\[ \<\n f, \n \rho\> = -\<V^\top, A\> =  - A( V^\top, \cdot) \]
and
\[ \Delta \rho  = \n^\bot H = \n^\bot V^\bot = - A(V^\top, \cdot).\]
This verifies the HRF soliton equations~\eqref{eq:HRF-soliton} for translating solitons.

\emph{Part 2: HRF Soliton $\implies$ MCF Soliton}\\
Assume $(g, \rho)$ satisfies (\ref{eq:HRF-soliton}) for \textit{some} $f$ and $\lambda \in \{-1, 0, 1\}$. Substituting the intrinsic quantities via (\ref{eq:gauss_id}), the system (\ref{eq:HRF-soliton}) is equivalent to:
\begin{equation} \label{eq:inverse_system}
    \begin{cases}
      \langle H, A \rangle + \n^2 f = \lambda g, \\
      \nabla^\perp {H} = A(\nabla f, \cdot).
    \end{cases}
\end{equation}

Define the ambient vector field
\[ V =  H - \nabla f + \lambda X. \]
We compute the ambient covariant derivative in the direction of a tangent vector $e_i$:
\begin{align*}
    \bar{\nabla}_{e_i} V &= \bar{\nabla}_{e_i}H - \bar{\nabla}_{e_i}\nabla f + \lambda \bar{\nabla}_{e_i} X \\
    &= (-A_{H}e_i + \nabla^\perp_i H) - (\nabla_i \nabla f + A(e_i, \nabla f)) + \lambda e_i.
\end{align*}
Applying (\ref{eq:inverse_system}) to replace $\nabla^2 f$ and $\nabla^\perp \vec{H}$:
\[ \bar{\nabla}_{e_i} V = - A_{H}e_i + A(\nabla f, e_i) - (\lambda e_i - A_{H}e_i) - A(e_i, \nabla f) + \lambda e_i = 0. \]
Thus $V$ is a constant vector field in $M$.

Now, if $\lambda \neq 0$, by translating the origin such that $V=0$, we obtain
\[ H - \nabla f + \lambda X = 0.\]
Projection to the normal bundle recovers the shrinker/expander equation \eqref{eq:MCF-soliton-shrinking}.
If $\lambda = 0$, we have
\[ H - \nabla f = V.\]
Projection gives the translator equation~\eqref{eq:MCF-soliton-translating}.
\end{proof}

\begin{rmk}
As a special case of the translating soliton, a static solution of MCF (SMCF) with $H=0$, i.e. a minimal (maximal) submanifold in $\R^{n+k}$ ($\R^{n,k}$), corresponds to a steady soliton of HRF with $f=0$, which satisfies
\begin{equation}\label{eq:harmonic-einstein}
    \begin{cases}
        \Ric - \alpha ~d\rho \otimes d\rho = 0 \\
        \Delta \rho = 0.
    \end{cases}
\end{equation}
Equation~\eqref{eq:harmonic-einstein} is called harmonic-Eienstein equations, which was studied by ~\cite{Xu2012, GeJiang2018}
\end{rmk}

\subsection{Monotonicity formula}

As an application of the correspondence between MCF and HRF, we can transport the monotonicity formulas of M\"uller \cite{Muller2012} to the SMCF. 

\begin{proof}[Proof of Theorem~\ref{thm:monotonicity}]
M\"uller \cite{Muller2012} proved that for a general HRF defined on a compact manifold $\Si$, the functional 
\[ \mathcal{F} := \int_\Si\(|\nabla f|^2 + R -\alpha |\nabla \rho|^2\)e^{-f}dV_g \]
evolves as
\begin{equation} \label{eq:muller_evol}
    \frac{d}{dt}\mathcal{F} = \int_\Sigma \left( 2|\Ric - \alpha d \rho \otimes d \rho + \nabla^2 f|^2 + 2\alpha |\tau_g(\rho) - \langle \nabla f, \nabla \rho \rangle|^2 \right) e^{-f} dV_g.
\end{equation}

The identity~\eqref{eq:muller_evol} naturally extends to complete case provided that the $\mathcal{F}$-functional is integrable and integration by parts works. Then the theorem follows directly by translating the intrinsic curvature terms and the derivatives of $\rho$ into extrinsic quantities through the Gauss-Codazzi equations (\ref{e:gauss-1}-\ref{e:tau-rho}).
\end{proof}

Note that for complete, non-compact $\Sigma$, the integrability of $\mathcal{F}$ is not guaranteed without additional assumptions. For a complete submanifold with bounded curvature with at most polynomial volume growth, the exponential decay of $e^{-f}$ is sufficient to dominate. On the other hand, spacelike submanifolds in pseudo-Euclidean space are particularly well-behaved, since spacelike condition provides a natural gradient bound. In this setting, the $\mathcal{F}$-functional is well-defined in the class of solutions with prescribed growth at infinity.

Actually, besides the above $\mathcal{F}$-functional, we can also define the $\mathcal{W}$-entropy as in \cite{Muller2012}, which is also non-decreasing for SMCF. Moreover, the $\mathcal{W}$-entropy is constant and coincides with Huisken's monotonicity formula~\cite{Hui1990} for solitons. This correspondence will be further explored in a future work.

\section{Appendix}

\subsection{Volume-preserving version of Theorem \ref{thm:main}}
When the submanifold $\Si$ is closed and $M=\R^{n+k}$, the correspondence of Theorem~\ref{thm:main} also holds between the volume-preserving MCF and the volume-preserving HRF.

First recall the volume-preserving MCF introduced in \cite{Hui1990}, which is obtained from the ordinary MCF by a suitable rescaling. More precisely, let
$$\tilde X(x,t)=\psi(t)X(x,t),$$
where the scaling factor $\psi(t)$ is chosen so that the volume of $\Sigma^n$ equipped with pullback metric of $\tilde X(x,t)$ is normalized to $1$ constant along the flow. The corresponding rescaled time variable is defined by
$$\tilde t(t)=\int_0^t\psi(\tau)^2 d\tau.$$
Then the rescaled immersion $\tilde X(\tilde t)$ satisfies
\begin{equation*}
\partial_{\tilde t}\tilde X =
\tilde H+\frac{1}{n} h\,\tilde X,
\end{equation*}
where
\begin{equation*}
h= \frac{1}{\vol_{g(t)}(\Sigma)}
\int_\Sigma  H^2\,d\mu.
\end{equation*}

The volume-preserving HRF was introduced in \cite{Muller2012}. Let $(g(t),\phi(t))$ be a solution to the HRF on a closed $n$-dimensional manifold $\Sigma$. Rescale the metric by
\[ \tilde g(t):=\lambda(t)g(t), \]
so that the volume with respect to $\tilde{g}(t)$ is normalized to $1$ along the flow. Let
$$\tilde t(t):=\int_0^t \lambda(s)\,ds.$$
Then the rescaled pair $(\tilde g(\tilde t),\phi(\tilde t))$ satisfies the volume-preserving HRF
\begin{equation*}\label{eq:normalized-rhf}
\begin{cases}
\partial_{\tilde t}\tilde g
=
-2\Ric_{\tilde g}
+2\alpha\nabla\phi\otimes\nabla\phi
+\dfrac{2}{n}r\,\tilde g,\\
\partial_{\tilde t}\phi
=
\tau_{\tilde g}\phi,
\end{cases}
\end{equation*}
where $r$ denotes the average of the modified scalar curvature:
$$r=\frac{1}{\vol_{g(t)}(\Sigma)}\int_{\Sigma}\bigl(R-\alpha|\nabla\phi|^2\bigr)\,d\vol_g.$$

Since the volume-preserving MCF and the volume-preserving HRF only differ from the original flows by a rescaling, which normalizes the volume, the correspondence of Theorem~\ref{thm:main} remains valid. In fact, letting $\ld(t)=\psi(t)^2$, we get the following theorem.

\begin{thm}\label{thm:pre-vol-main}
Suppose $X_0:\Si \to \R^{n+k}$ is a closed immersed submanifold with Gauss map $\rho_0$ and pull-back metric $g_0$. Then the volume-preserving MCF 
\begin{equation*}
\begin{cases}
    \partial_{\tilde t}\tilde X= \tilde H + \frac{1}{n}\tilde h\,\tilde X, \\
\tilde{X}(0) = X_{0}.
\end{cases}
\end{equation*}
and the volume-preserving HRF
\begin{equation*}
\begin{cases}
\partial_{\tilde t}\tilde g
=-2\Ric_{\tilde g} - 2 \nabla \phi \otimes \nabla \phi+ \dfrac{2}{n}r\,\tilde g,\\
\partial_{\tilde t}\phi = \tau_{\tilde g}\phi,\\
\tilde g(0) = g_{0}, \quad \tilde \rho(0) = \rho_{0},
\end{cases}
\end{equation*}
are equivalent.
\end{thm}

\subsection{Harmonic-Ricci-type equations in a curved ambient space}

In this subsection, we show that in a general Riemannian or pseudo-Riemannian ambient space, the pullback metric and the Gauss map of the MCF  also satisfies a harmonic-Ricci-type system. The evolution of the Gauss map of the MCF in a Riemannian manifold was studied by Wang \cite{Wang2003}. 

Let $(M, h)$ be an $(n+k)$-dimensional Riemannian or pseudo-Riemannian manifold and $\Si$ be an $n$-dimensional manifold. Suppose $X:\Sigma \times [0,T)\to M$
is a solution to the MCF starting at an immersion $X_0$. When $M$ is pseudo-Riemannian, we require $X$ to be space-like for all $t\in [0,T)$. This can be guaranteed if $X_0$ is a space-like graph in a product space, see~\cite{EH1991} for example. 

Now consider the canonical Grassmannian bundle
$$G_M := \left\{ \text{oriented $n$-dimensional (spacelike) subspaces of } T_yM|\ y\in M\right\}\xrightarrow{p} M, $$
whose fiber space is isometric to the usual Grassmannian manifold $G$.
The bundle-valued Gauss map of $X$, $$\bar\rho:\Sigma\times[0,T)\to G_M,$$  is defined by sending each point $(x,t)$ to the tangent space $dX(T_x\Sigma)$.

Observe that 
$$X(x,t)=p\circ\bar\rho(x,t),$$
hence the Gauss map defined in this way already contains the position data of the immersion $X$. In particular, when $M$ is either $\mathbb{R}^{n+k}$ or $\mathbb{R}^{n,k}$, the Grassmannian bundle is trivial, namely $G_M=M\times G$. In this case, one may recover the Gauss map used in Section 2 by forgetting the position variable.

By representing each subspace as an unit $n$-product, we may embedd the Grassmannian bundle $G_M$ into the $n$-wedge-product bundle $\wedge^n M$. Then the Gauss map can be viewed as a section in the pull-back bundle $X^*(\wedge^n M)$, by sending each point $(x,t)$ to $e_1\wedge\cdots\wedge e_n$, where $\{e_1, \cdots, e_n\}$ is an oriented orthonormal basis of $dX(T_x\Sigma)$. Thus we can define a modified  differential of $\b\rho$ by
$$\b d\b\rho(\xi)=\b\nabla^h_{dX(\xi)}(e_1\wedge\cdots\wedge e_n),$$
where $\b\nabla^h$ is the canonical connection on $\wedge^nM$ induced by the Levi-Civita connection on $M$. 

\begin{thm}\label{thm:app}
Let $X:[0,T)\times\Sigma\to M$ be a solution to the MCF starting at an immersed submanifold  $X_0:\Sigma \to M$. Assume that $X$ space-like for all $t\in [0, T)$ if $M$ is pseudo-Riemannian. Then the corresponding pullback metric and  Gauss map $(g(t), \b\rho(t))$ satisfy the following harmonic-Ricci-type system:
\begin{equation}\label{eq:general-evolution}
\begin{cases}
\partial_t g = -2\Ric_g - 2\langle \b d\b\rho\otimes \b d\b\rho\rangle_{\b h} + 2X^*\Ric_h^\top,\\
\partial_t \b\rho = \tau_g(\b\rho) + \Ric_h^{*\bot},
\end{cases}
\end{equation}
where
$$2X^*\Ric_h^\top(\partial_i,\partial_j)=R_h(dX(\partial_i),dX(\partial_l),dX(\partial_l),dX(\partial_j)),$$
$$h(\Ric_h^{*\bot}(e_i),v)=R^h(e_i,e_l,e_l,v),$$
and 
$$\langle \b d\b\rho\otimes \b d\b\rho\rangle_{\b h}(\partial_i, \partial_j) = \b h \bigl(\b d\b\rho(\partial_i),\b d\b\rho(\partial_j)\bigr).$$

Moreover, when $M$ is $\mathbb{R}^{n+k}$ or $\mathbb{R}^{n,k}$ (and $X_0$ is spacelike),  the evolution of $(g(t),\rho(t))$ reduces to the HRF \eqref{eq:hrf-gauss} with coupling constant $\alpha=-\ep$.
\end{thm}

Before the proof, we introduce the canonical tautological and cotautological bundles
$$\G_M^\top\xrightarrow{p_1}G_M,\qquad \G_M^\bot\xrightarrow{p_2}G_M.$$
The fiber of $\G_M^\top$ at $\sigma(y)\in G_M$ is the $n$-dimensional subspace $\sigma(y)$ itself, and the fiber of $\G_M^\bot$ at $\sigma(y)$ is the orthogonal complement of $\sigma(y)$. These bundles carry canonical metrics and connections, denoted by
$(\G_M^\top,h^\top,\nabla^\top)$  and $(\G_M^\bot,h^\bot,\nabla^\bot)$.

By the definition of $\b d\b\rho$, 
\begin{equation}\label{al:drho}
\b d\b\rho(\xi) = \b\nabla^h_{dX(\xi)}(e_1\wedge\cdots\wedge e_n) \notag = h\bigl(\nabla^\bot_\xi e_i,v_\alpha\bigr)\, e_1\wedge\cdots\wedge v_\alpha\wedge\cdots\wedge e_n .
\end{equation}
We denote $E_{i\alpha}=e_1\wedge\cdots\wedge v_\alpha\wedge\cdots\wedge e_n$, where $v_\alpha$ is placed in the $i$-th factor. There is one-to-one corresponding between $E_{i\alpha}$ and $e_i^*\otimes v^\alpha$ which gives a frame of $\b\rho^*(\G_M^{\top *}\otimes \G_M^\bot)$. Thus 
$\b d\b\rho\in \Gamma\bigl(T^*\tilde\Sigma\otimes \b\rho^*(\G_M^{\top *}\otimes \G_M^\bot)\bigr)$ and 
under a local frame $\{\partial_i,\partial_t\}$ of $T(\tilde{\Sigma})$,
$$\b d\b\rho=\nabla_k\b\rho_l^\alpha\,\partial_k^*\otimes e_l^*\otimes v_\alpha, \qquad
X^*\b d\b\rho =\nabla_k\b\rho_l^\alpha X_j^l\,\partial_k^*\otimes\partial_j^*\otimes v_\alpha,$$
where $X_j^l=h(dX(\partial_j),e_l)$.

Let us recall the second fundamental form of a subbundle. Since
$$X^*(TM)=\b\rho^*(\G_M^\top)\oplus\b\rho^*(\G_M^\bot),$$
we may regard $\b\rho^*(\G_M^\top)\to\Sigma$ as a subbundle of $X^*(TM)\to\Sigma$. Its second fundamental form is defined by
\begin{equation}\label{eq:second-form}
B_{\b\rho^*(\G_M^\top)}(\xi,e):=\pi^\perp\circ\nabla_{dX(\xi)}e,
\end{equation}
and 
$$B_{\b\rho^*(\G_M^\top)}\in \Gamma\bigl(T^*\tilde{\Sigma}\otimes \b\rho^*(\G_M^{\top *}\otimes \G_M^\bot)\bigr)$$
can be viewed as a section of the same bundle as $\b d\b\rho$. In a local frame, we write
$$B_{\b\rho^*(\G_M^\top)}=h_{il}^\alpha\,\partial_i^*\otimes e_l^*\otimes v_\alpha.$$
By \eqref{al:drho} and \eqref{eq:second-form}, we obtain
$$\nabla_i\b\rho_l^\alpha=h_{il}^\alpha.$$

Let $A\in\Gamma(T^*\tilde{\Sigma}\otimes T^*\Sigma\otimes N\b \Sigma)$ be the second fundamental form of $X$. Since $dX$ can be regarded as a bundle isomorphism from $T\Sigma$ to $\b\rho^*\G_M^\top$, the second fundamental forms of the two subbundles satisfy, 
$$A_{ij}^\alpha=h_{il}^\alpha X_j^l.$$ 
Consequently, 
\begin{equation}\label{eq:rho-sec}
A=X^*\b d\b\rho.
\end{equation}

\begin{proof}[Proof of Theorem \ref{thm:app}]
We first compute
\begin{align*}
\partial_t g(\partial_i,\partial_j)= -2h(H,A(\partial_i,\partial_j)).
\end{align*}
Taking the trace of the Gauss equation with respect to the frame $\{\partial_l\}$, we obtain
\begin{equation}\label{eq:gauss}
X^*\Ric_h^\top(\partial_i,\partial_j)=\Ric_g(\partial_i,\partial_j)+h(A(\partial_i,\partial_l),A(\partial_j,\partial_l))-h(A(\partial_i,\partial_j),H).
\end{equation}
And by \eqref{eq:rho-sec}, we have 
$$\langle \b d\b\rho\otimes \b d\b\rho\rangle_{\b h,ij}=h(A_{il},A_{jl}).$$
This gives the first equation of \eqref{eq:general-evolution}.

We next compute the evolution of the Gauss map. A direct calculation gives
\begin{equation*}
\partial_t\b\rho = \b d\b\rho(\partial_t) = h(\nabla^\bot_{e_i}H,v_\alpha)e_i^*\otimes v_\alpha.
\end{equation*}
Taking the trace of the Codazzi equation with respect to the frame $\{e_l\}$ yields
$$h(\Ric_h^{*\bot}(e_i),v) +h((\nabla_{\partial_l}A)(\partial_l,dX^{-1}(e_i)),v)=h(\nabla^\bot_{e_i}H,v).$$
By \eqref{eq:rho-sec},  we obtain the second equation of \eqref{eq:general-evolution}.

Finally, when $M$ is flat, all the curvature terms in \eqref{eq:general-evolution} vanishes and the bundle-valued Gauss map $\b{\rho}:\Si\to G_M$ reduces to the usual Gauss map $\rho:\Si\to G$. Then it is easy to see that \eqref{eq:general-evolution} reduces to the HRF \eqref{eq:hrf-gauss}.

\end{proof}


\end{document}